\documentclass[11 pt,reqno]{amsart}
\usepackage{amsmath}
\usepackage{amsfonts}
\usepackage{amssymb}
\usepackage{mathrsfs}
\usepackage{amsthm}
\usepackage{multicol, verbatim}
\newcommand\C{\mathbb{C}}

\newcommand\R{\mathbb{R}}

\newcommand\D{\mathcal{D}}

\usepackage{hyperref}
\usepackage[dvipsnames]{xcolor}
\hypersetup{
    colorlinks,
    linkcolor={red!50!black},
    citecolor={blue!50!black},
    urlcolor={blue!80!black}
}
\usepackage{graphicx}
\graphicspath{ {./images/} }

\numberwithin{equation}{section}

\allowdisplaybreaks

\theoremstyle{remark}
\newtheorem{remark}{Remark}
\theoremstyle{lemma}
\newtheorem{lemma}{Lemma}
\theoremstyle{theorem}
\newtheorem{theorem}{Theorem}
\theoremstyle{corollary}

\theoremstyle{proposition}
\newtheorem{proposition}{Proposition}
\usepackage{amsaddr}
\usepackage{mathtools}

\DeclarePairedDelimiter\floor{\lfloor}{\rfloor}

\title{Partitions into Distinct Parts with Bounded Largest Part}
\date{}
\author{Walter Bridges}
\address{Louisiana State University \\ Department of Mathematics \\ Baton Rouge, Louisiana,  USA}

\email{wbridg6@lsu.edu}

 \keywords{partitions, asymptotic analysis, circle method, probability}
 
 \subjclass[2020]{05A17 11P82}

\begin{document}

\maketitle

\begin{abstract}
We prove an asymptotic formula for the number of partitions of $n$ into distinct parts where the largest part is at most $t\sqrt{n}$ for fixed $t \in \R$.   Our method follows a probabilistic approach of Romik, who gave a simpler proof of Szekeres' asymptotic formula for distinct parts partitions when instead the number of parts is bounded by $t\sqrt{n}$.  Although equivalent to a circle method/saddle-point method calculation, the probabilistic approach motivates some of the more technical steps and even predicts the shape of the asymptotic formula, to some degree.
\end{abstract}

\section{Introduction}

A {\it distinct parts partition} $\lambda$ of $n$ is a set of positive integers $\{\lambda_1, \dots, \lambda_{\ell}\}$ satisfying
$$
 \lambda_1 > \lambda_2 > \dots > \lambda_{\ell} > 0; \qquad \sum_{j=1}^{\ell} \lambda_j = n.
$$  Here, $|\lambda|=n$ is the {\it size} of $\lambda$, and the $\lambda_j$ are its parts.  For example, the distinct parts partitions of $5$ are $5,$ $4+1,$ and $3+2$.  Let $d(n)$ denote the number of distinct parts partitions of $n$.  These numbers are easily seen to be generated by the following infinite product
$$
\sum_{n \geq 0}d(n) x^n = \prod_{k \geq 1} (1+x^k).
$$
Pioneering work of Hardy and Ramanujan used the modular properties of the infinite product to obtain an asymptotic series for $d(n)$ (and similar enumerations) after representing these coefficients as contour integrals around the origin (\cite{HR}, $\S 7.1$).  The main term in Hardy and Ramanujan's asymptotic series is
\begin{equation}\label{E:HRasymp}
d(n) \sim \frac{1}{4\sqrt[4]{3} n^{\frac{3}{4}}}e^{\frac{\pi}{\sqrt{3}}\sqrt{n}}.
\end{equation}
The {\it circle method} is now often used as an umbrella term for the asymptotic analysis of contour integrals, including Hardy-Ramanjuan's method and its many variants, as well as certain cases of the {\it saddle-point method}.  For an exposition of Hardy, Ramanujan, and Rademacher's original work, see \cite{A} Ch. 5-6 and for the saddle-point method, see \cite{FS} Ch. VIII.

A more recent approach to these asymptotic statistics, begun by Fristedt in \cite{F} and used by Romik in \cite{R}, is to reformulate the proof using probability theory.  This can make some of the steps more intuitive.  We explain these ideas further in Section \ref{S:Outline}.

Let $t$ be a fixed positive real number.  We study a restriction of $d(n)$ defined as
$$
d_{t}(n):=\text{Coeff} \ [x^n] \ \mathcal{D}_{t,n}(x), \qquad \text{where} \qquad \D_{t,n}(x):= \prod_{k \leq t\sqrt{n}} (1+x^k).
$$
Thus, $d_t(n)$ is the number of distinct parts partitions of $n$ with largest part at most $t\sqrt{n}$.  The smallest possible largest part in a distinct parts partition of $n$ with largest part at most $t\sqrt{n}$ is $\kappa$, where
$$
1+2+\dots + (\kappa-1) = \frac{\kappa(\kappa-1)}{2} < n \leq \frac{\kappa(\kappa+1)}{2}.
$$
Thus, we ignore the range $t\leq \sqrt{2}$, where often $d_t(n)=0$, and consider only $t> \sqrt{2}$.  We prove the following asymptotic formula for $d_t(n)$.  Here and throughout, $\floor{\alpha}$ denotes the greatest integer less than or equal to $\alpha$ and $\{\alpha\}:= \alpha - \floor{\alpha}$.

\begin{theorem}\label{T:dtasymp}  Let $t > \sqrt{2}$.  Define $\beta:\left(\sqrt{2}, \infty\right)\to \left(-\infty, \frac{\pi}{2\sqrt{3}}\right)$ implicitly as a function of $t$ so that
\begin{equation}\label{E:betadef}
1=\int_0^t \frac{ue^{-\beta u}}{1+e^{-\beta u}}du.
\end{equation}
(See Proposition \ref{P:beta}.)  Let
\begin{equation}\label{E:B(t)andA_n(t)def}
B(t):= 2\beta + t \log \left(1+e^{-\beta t}\right) \qquad \text{and} \qquad A_n(t):= \frac{e^{\frac{\beta t}{2}}+e^{-\frac{\beta t}{2}}}{2\left(1+e^{-\beta t}\right)^{\{t\sqrt{n}\}}}\sqrt{\frac{\beta'(t)}{\pi t}}.
\end{equation}
Then
$$
d_t(n) \sim \frac{A_n(t)}{n^{3/4}}e^{B(t)\sqrt{n}}.
$$
\end{theorem}

The oscillatory factor $\left(1+e^{-\beta t}\right)^{-\{t \sqrt{n}\}}$ is present because $t\sqrt{n}$ is not always an integer.  Numerically, this oscillation is also reflected in $d_t(n)$, which appears not to be increasing for $t$ close to $\sqrt{2}$.

\begin{remark}\label{R:betaBAlimits}
We record properties of the functions $\beta(t), B(t)$ and $A(t):=A_n(t) \left(1+e^{-\beta t}\right)^{\{t\sqrt{n}\}}$ in Section \ref{S:functionproperties}.  In particular, we show that $\beta$ and $B$ are strictly increasing, and we show that $\beta(t)$, $B(t)$ and $A(t)$ tend to $\frac{\pi}{2\sqrt{3}}, \frac{\pi}{\sqrt{3}}$ and $ \frac{1}{4\sqrt[4]{3}}$ respectively, as $t \to \infty$.  Thus, Theorem \ref{T:dtasymp} is consistent with Hardy and Ramanujan's asymptotic formula, and we recover \eqref{E:HRasymp} by letting $t \to \infty$.
\end{remark}

\begin{remark}\label{R:dt0percent}
It has been shown that the largest part of a typical distinct parts partition of $n$ is $c\sqrt{n} \log n$ for some $c$ (\cite{F}, Thm. 9.4), so that our $d_t(n)$ counts (asymptotically) $0\%$ of distinct parts partitions of $n$ for any fixed $t$.  That is, we have $d_t(n) = o\left(d(n)\right)$.  This is also implied by Theorem 1 because $B(t)$ is strictly increasing to $\frac{\pi}{\sqrt{3}}$ (see Proposition \ref{P:B(t)andA(t)limits}).  
\end{remark}

\begin{remark}
A weak form of Theorem \ref{T:dtasymp} was used recently in the author's proof of limit shapes for unimodal sequences \cite{B}, and the methods of \cite{B} suggest that a limit shape for the partitions enumerated by $d_t(n)$ is
$$
f_t(x):= \frac{1}{\beta} \log\left(\frac{1+e^{-\beta x}}{1+e^{-\beta t}}\right),
$$
(see Figure \ref{F:limitshape}).  It is curious that the concavity of the curves $y=f_t(x)$ changes at $t=2$.
\begin{figure} \begin{center} \includegraphics[width=.9\textwidth]{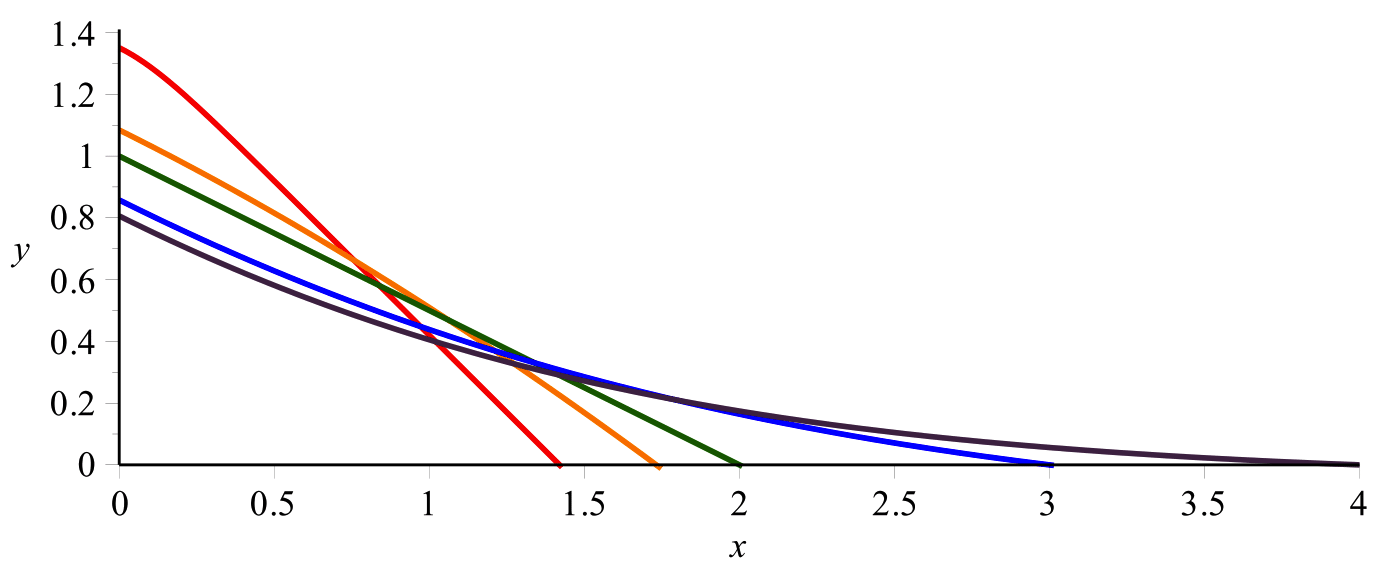} \end{center} \caption{Plots of $y=f_t(x)$ when $t=1.42$, $\sqrt{3}$, $2$, $3$ and $4$.  The $x$-intercepts occur at $t$. (generated using Maple)} \label{F:limitshape}\end{figure}
\end{remark}

Szekeres found an asymptotic formula for distinct parts partitions when instead the {\it number of parts} is at most $t\sqrt{n}$ (\cite{S1}, \cite{S2}).  When parts are allowed to repeat, bounding the number of parts and bounding the size of the largest part give the same enumeration function due to a simple symmetry on the Ferrer's diagrams of partitions called conjugation (see \cite{A}, $\S 1.3$).  But here, when parts are distinct, these two notions are different.

Szekeres' proof in \cite{S2} is based on the saddle-point method, and later Romik \cite{R} recast and simplified this proof using Fristedt's probabilistic machinery \cite{F}.  Although equivalent to a circle method calculation, Romik's proof motivates some of the more technical steps in the proof and even predicts the shape of the asymptotic formula, to some degree.  Our proof here closely follows Romik.

In Section \ref{S:Outline}, we outline the proof, motivating the probabilistic model.  Then we state three propositions that together imply Theorem \ref{T:dtasymp}.  In Section \ref{S:functionproperties} we record some properties of the functions $\beta(t)$, $B(t)$ and $A(t)$, including those mentioned in Remark \ref{R:betaBAlimits}.  In Section \ref{S:Proofs}, we prove Propositions \ref{P:logDtasymp}, \ref{P:expvar} and \ref{P:Probasymp}.  Section \ref{S:lemmaproofs} provides the proofs of two technical lemmas used in Section \ref{S:Proofs}; these could be useful in similar asymptotic analysis and may be of independent interest.

\section*{Acknowledgments} We thank the two anonymous referees their helpful suggestions.

\section{Proof Outline and Probabilistic Model}\label{S:Outline}

Throughout the remainder of the paper, $x$ will be a positive real number.  We prove Theorem \ref{T:dtasymp} through three propositions.  Proposition \ref{P:logDtasymp} anticipates the asymptotic behavior of $\log d_{t}(n)$ through classical saddle-point bounds, while Propositions \ref{P:expvar} and \ref{P:Probasymp} complete the proof using Fristedt's probabilistic machinery.

\subsection{Saddle-Point Bounds}\label{S:SPB}

  We begin with the trivial inequality,
\begin{equation}\label{E:saddlepointbound}
d_t(n) \leq x^{-n}\D_{t,n}(x).
\end{equation}
As explained in the book of Flajolet and Sedgewick (\cite{FS}, p. 550), the right-hand side of \eqref{E:saddlepointbound}, as a function of $x \in (0, \infty)$, has positive second derivative with respect to $x$ and tends to $+\infty$ when $x \to 0$ and also when $x \to \infty$.  Thus, there is a unique saddle-point $x=x_n$ on the positive real axis for the function $|z^{-n}\D_{t,n}(z)|$ of a complex variable $z$.  In fact, $x$ will approach 1 as $n \to \infty$, from below when $t>2$ and from above when $t<2$.  As in many similar cases, we anticipate that this $x$ forces the $n$-th term of the right-hand side of \eqref{E:saddlepointbound} to dominate, so that we expect the logarithm of the two terms to be asymptotic:
\begin{equation}\label{E:logdtstuff}
\log d_t(n) \sim \log \left(x^{-n} \D_{t,n}(x) \right).
\end{equation}
Thus, with this $x$ in hand, we ascertain an upper bound for $d_t(n)$ by finding the asymptotic behavior of the right-hand side of \eqref{E:logdtstuff}.  

More explicitly, we set $x= e^{-\frac{y}{\sqrt{n}}}$ for $y \in \mathbb{R}$ and write
$$
\log\left(x^{-n}\D_{t,n}(x)\right)= \sqrt{n}f_n(y), \qquad \text{where} \qquad f_n(y):= y+ \frac{1}{\sqrt{n}} \log \D_{t,n}\left(x\right).$$
One computes
\begin{equation}\label{E:fnprime}
f'_n(y)= \frac{d}{dy} f_n(y) = 1-\frac{1}{n} \sum_{k \leq t\sqrt{n}} \frac{kx^k}{1+x^k},
\end{equation}
so that the saddle-point occurs at (or very near) $x$ when $f'_n(y) \sim 0.$  We will show in Proposition \ref{P:expvar} that this is accomplished by choosing $y=\beta$.  Indeed, the sum in \eqref{E:fnprime} is just a Riemann sum for the integral \eqref{E:betadef} defining $\beta$.  With $\beta$ in hand, an application of Euler-Maclaurin summation leads to the following.

\begin{proposition}\label{P:logDtasymp}
With $x=e^{-\frac{\beta}{\sqrt{n}}}$, we have
\begin{equation}\label{E:logDtasymp}
\log \left(x^{-n}\D_{t,n}(x)\right) = B(t)\sqrt{n} + \log\left(\sqrt{\frac{1+e^{-\beta t}}{2}}\right)-\log\left(1+e^{-\beta t}\right) \{t\sqrt{n}\} + o(1),
\end{equation}
where $B(t)$ is defined in Theorem \ref{T:dtasymp}.
\end{proposition}

As observed above, Proposition \ref{P:logDtasymp} implies $\log d_t(n) \ll B(t)\sqrt{n}$, but we will see later that the two are actually asymptotic.

\subsection{Probabilistic Model}

From probability theory we will require the elementary notions of expectation, variance and distribution of discrete random variables, as well Fourier inversion of characteristic functions (which in this context is equivalent to an application of Cauchy's Theorem from complex analysis).  We will also mention central and local limit theorems.  All of these topics are covered in most standard probability texts; for instance see \cite{Bi}.

We now repair the inequality \eqref{E:saddlepointbound} by introducing a certain probability measure depending on $x$ as
\begin{equation}\label{E:repairineq}
P_{x,n}(N=k)= \frac{d_t(k)x^k}{\D_{t,n}(x)}, \qquad \text{so that} \qquad d_t(n)= x^{-n}\D_{t,n}(x)P_{x,n}(N=n).
\end{equation}
We define $P_{x,n}$ and the random variable $N$ below.

Our probability measure $P_{x,n}$ is similar to the ones introduced by Fristedt \cite{F}, who invented an early variant of a {\it Boltzmann model} for partitions and used it to prove many far-reaching results on the structure of partitions.  When applied to partitions, Boltzmann sampling algorithms select partitions of size roughly $n$, roughly uniformly and in nearly linear time, assuming $n$ is large.  See \cite{DFLS} for more on Boltzmann sampling for combinatorial structures.

Following Fristedt, we define a probability measure $P_{x,n}$ on the set of partitions $\lambda$ generated by $\D_{t,n}$ by setting
$$P_{x,n}(\lambda):= \frac{x^{|\lambda|}}{\D_{t,n}(x)},$$
where $|\lambda|$ is the size of the partition $\lambda$, i.e. the sum of its parts.

Let $\{X_k\}_{k=1}^{t\sqrt{n}}$ be random variables giving the multiplicity of $k$ in a partition $\lambda$.  Since our partitions have distinct parts, $X_k$ is Bernoulli and one computes
$$
P_{x,n}(X_k=0) = \frac{1}{1+x^k} \qquad \text{and} \qquad P_{x,n}(X_k=1) = \frac{x^{k}}{1+x^k}.
$$
It is also straightforward to show that the $X_k$'s are independent under $P_{x,n}$.  Now set \newline $N:= \sum_{k \leq t\sqrt{n}} kX_k$, a random variable representing the size of a partition.  Using independence, its expectation and variance under $P_{x,n}$ are
\begin{equation}\label{E:expvardef}
{\rm E}_{x,n}(N)=\sum_{k \leq t\sqrt{n}} \frac{kx^k}{1+x^k}, \qquad \sigma_n^2:= {\rm Var}_{x,n}(N)= \sum_{k \leq t\sqrt{n}}\frac{k^2x^{k}}{(1+x^k)^2}.
\end{equation}

Returning to \eqref{E:fnprime}, we see that $f_n'(y) \sim 0$ if and only if ${\rm E}_{x,n}(N) \sim n$, so the choice $y=\beta$ ensures that the expectation of $N$ is asymptotically $n$ under $P_{x,n}$ with $x=e^{-\frac{\beta}{\sqrt{n}}}$.  Thus, we prove the following.

\begin{proposition}\label{P:expvar}
With $x=e^{-\frac{\beta}{\sqrt{n}}}$, we have
\begin{equation}\label{E:expasymp}
{\rm E}_{x,n}(N) = n + O(\sqrt{n}),
\end{equation}
and
\begin{equation}\label{E:varasymp}
\sigma^2_n= {\rm Var}_{x,n}(N)= \frac{t}{(1+e^{\beta t})\beta'(t)} n^{\frac{3}{2}}+ O(n).
\end{equation}
\end{proposition}

In fact, we will show that $\frac{N-n}{\sigma_n}$ is asymptotically normally distributed under $P_{x,n}$ (see Figure \ref{F:Px}), and so a sort of central limit theorem holds for the $X_k$.    Heuristically, this suggests that $P_{x,n}(N=n) \sim \frac{1}{\sqrt{2\pi} \sigma_n}$, as follows: $N$ takes only integer values, so we expect

\begin{figure} \begin{center} \includegraphics[width=.7\textwidth]{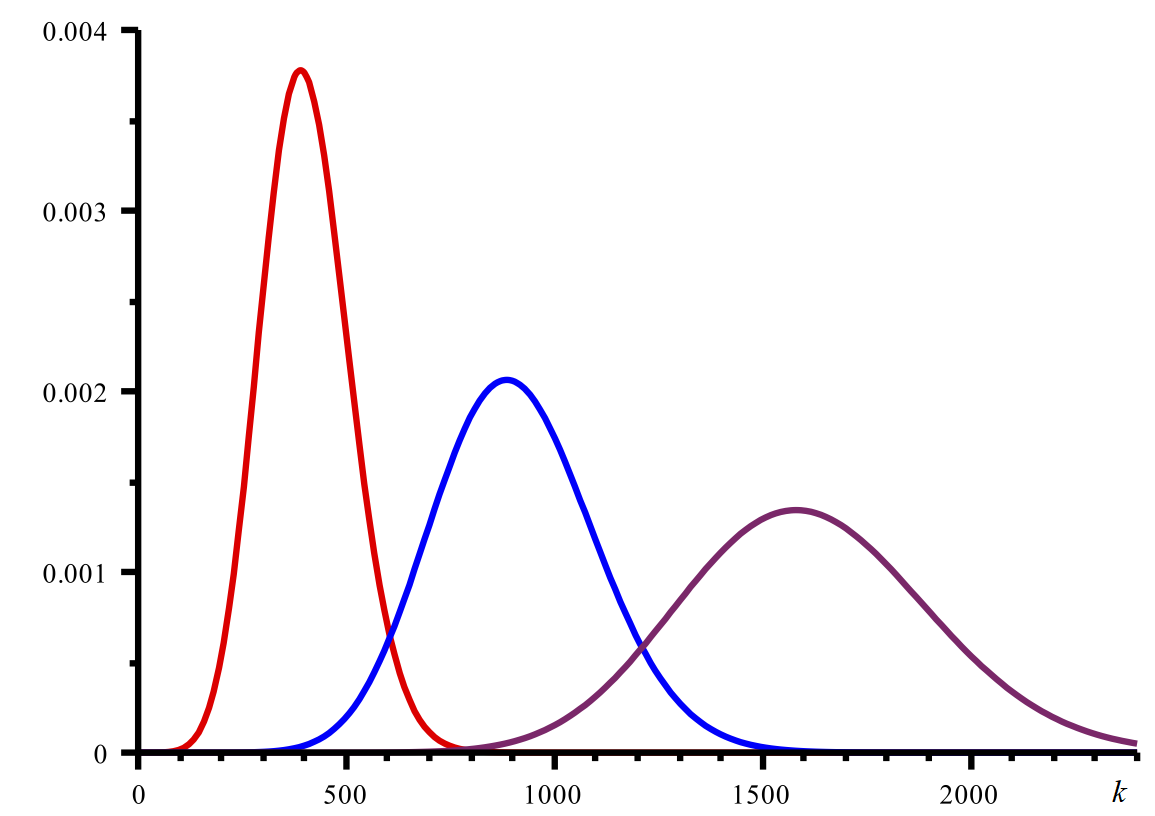} \end{center} \caption{Plots of $P_{x,n}(N=k)$ with $t=3$ and $n=$ {\color{red} 400}, {\color{blue} 900} and {\color{Plum} 1600}. (generated using Maple)} \label{F:Px} \end{figure}

$$
P_{x,n}(N=n) = P_{x,n}\left(-\frac{1}{2}\leq N-n\leq \frac{1}{2}\right)= P_{x,n}\left(-\frac{1}{2\sigma_n} \leq \frac{N-n}{\sigma_n} \leq \frac{1}{2\sigma_n}\right) $$ $$\approx \frac{1}{\sqrt{2\pi}}\int_{-\frac{1}{2\sigma_n}}^{\frac{1}{2\sigma_n}} e^{-\frac{u^2}{2}} du \sim \frac{1}{\sqrt{2\pi}\sigma_n}.
$$
Since $d_t(n)= x^{-n}\D_{t,n}(x)P_{x,n}(N=n)$, the local limit theorem suggested above, together with Proposition \ref{P:logDtasymp}, implies Theorem 1.  Our final proposition is a formal statement of the asymptotic normality of $\frac{N-n}{\sigma_n}$ together with the above heuristic.

\begin{proposition}\label{P:Probasymp}
With $x=e^{-\frac{\beta}{\sqrt{n}}}$, we have
\begin{equation}\label{E:Nnormal}
\lim_{n \to \infty} P_{x,n}\left( \frac{N-n}{\sigma_n} \leq v \right) = \frac{1}{\sqrt{2\pi}} \int_{-\infty}^v e^{-\frac{u^2}{2}} du, \qquad \text{for $v \in \mathbb{R}$.}
\end{equation}
Moreover,
\begin{equation}\label{E:Probasymp}
P_{x,n}(N=n) \sim \frac{1}{\sqrt{2\pi}\sigma_n}.
\end{equation}
\end{proposition}

The proof of Proposition \ref{P:Probasymp} proceeds via Fourier inversion of the characteristic function for $N$.

\section{The functions $\beta(t), B(t)$ and $A(t)$.}\label{S:functionproperties}

In this section, we prove the claimed limits in Remark \ref{R:betaBAlimits} and record some additional properties of the functions $\beta(t)$, $B(t)$ and $A(t):= A_n(t)\left(1+e^{-\beta t}\right)^{\{t\sqrt{n}\}}$.  Here and in later sections, we require properties of the dilogarithm function, Li$_2(z)$, defined for $z \in \C \setminus (-\infty,-1)$ by the integral 
\begin{equation}\label{E:Li2def}
{\rm Li}_2(z):= -\int_0^z \frac{\log(1-w)}{w}dw,
\end{equation}
 taking the principal branch of the complex logarithm.  We also have the Taylor expansion Li$_2(z)= \sum_{n \geq 1} \frac{z^{n}}{n^2}$ for $|z| \leq 1$, and hence Li$_2(1)=\frac{\pi^2}{6}$.  See \cite{AS}, \S 27.7, where \newline $f(x)=  {\rm Li}_2(1-x)$.

\begin{proposition}\label{P:beta}
The function $\beta=\beta(t)$ satisfies the following properties.
\begin{itemize}
\item[{\rm (a)}] The function $\beta$ is well-defined by \eqref{E:betadef}, and we have
$$
 \begin{cases} \beta < 0 & \text{for $\sqrt{2} < t< 2$,} \\ \beta =0 & \text{for $t=2$,} \\ \beta > 0 & \text{for $t > 2$.} \end{cases}
$$
\item[{\rm (b)}]  The function $\beta$ is strictly increasing with 
\begin{equation}\label{E:betaprimedef}
\beta'(t)=\begin{cases} \frac{\beta t}{2(1+e^{\beta t})-t^2} & \text{for $t\neq 2$,} \\ \frac{3}{2} & \text{for $t = 2$.} \end{cases}
\end{equation}

\item[{\rm (c)}]  The following limits hold: 
\begin{equation}\label{E:betalimits}
\lim_{t \to \sqrt{2}^+} \beta =-\infty, \qquad \lim_{t \to \infty} \beta =\frac{\pi}{2\sqrt{3}}.
\end{equation}
\end{itemize}
\end{proposition}
\begin{proof}
For fixed $t > \sqrt{2}$, set $g(y):= \int_0^t \frac{u}{1+e^{yu}} du$.  Then $g(y)$ is a decreasing function of $y$ with
$$
\lim_{y \to -\infty} g(y) = \frac{t^2}{2} > 1 \qquad \text{and} \qquad \lim_{y \to \infty} g(y) = 0.
$$
Hence, there exists a unique solution, $\beta=\beta(t)$, to the equation $g(y)=1$, and $\beta$ is well-defined by \eqref{E:betadef}.

If $t>2$, then we must have $\beta(t)>0$, for if not,
$$
1=\int_0^t \frac{u}{1+e^{\beta u}}du > \frac{1}{2} \int_0^t udu= \frac{t^2}{4},
$$
which leads to the contradiction $2>t$.  A similar argument proves the remaining statements in part (a).

For $t \neq 2$, we rewrite \eqref{E:betadef} as
\begin{equation}\label{E:beta^2int}
\beta^2= \int_0^{\beta t} \frac{u}{1+e^{u}}du,
\end{equation}
and take the derivative of both sides to get
$$
\beta'(t)= \frac{\beta t}{2(1+e^{\beta t})-t^2}, \qquad \text{for $t \neq 2.$}
$$
To find $\beta'(2)$, we use the first two terms of the Taylor series for the integrand in \eqref{E:beta^2int} to write
$$
\beta^2 = \frac{\beta^2 t^2}{4}- \frac{\beta^3t^3}{12} + O\left(\beta^5 t^5\right),
$$
for $t$ near 2 (so $\beta$ near 0).  This implies
$$
\beta = \frac{3}{t}-\frac{12}{t^3} + O(\beta^3t^2),
$$
and thus by L'Hospital's Rule,
$$
\beta'(2)= \lim_{t \to 2} \frac{\beta}{t-2}= \lim_{t \to 2} \frac{\frac{3}{t}-\frac{12}{t^3} + O(\beta^3t^2)}{t-2}= \lim_{t \to 2} \frac{-3}{t^2}+\frac{36}{t^4}= \frac{3}{2}.
$$  We see that $\beta'(t) > 0$ for $t>2$ by observing
\begin{equation}\label{betapos}
1=\int_0^t \frac{u}{1+e^{\beta u}}du > \frac{1}{1+e^{\beta t}}\int_0^t u du = \frac{1}{1+e^{\beta t}} \cdot \frac{t^2}{2}.
\end{equation}
A similar argument shows that $\beta'(t) > 0$ for $\sqrt{2} < t < 2$ also.  Thus, part (b) is proved.

The first limit in \eqref{E:betalimits} is easy to see, for
$$
1 = \int_0^t \frac{u}{1+e^{\beta u}}du \leq \int_0^t u du = \frac{t^2}{2},
$$
and thus as $t \to \sqrt{2}^+$, we must have $\beta \to -\infty$.  We evaluate the second limit in \eqref{E:betalimits} by expressing the integral in \eqref{E:beta^2int} in terms of the dilogarithm.  Thus, \eqref{E:beta^2int} implies that for $t >2$, we have
\begin{equation}\label{E:betadef2}
\beta^2=\int_0^{\beta t} \frac{u}{1+e^{u}}du= \text{Li}_2\left(1-e^{-\beta t}\right)- \frac{1}{2}\text{Li}_2\left(1-e^{-2\beta t}\right).
\end{equation}
Hence, $\lim_{t \to \infty} \beta^2= \frac{\pi^2}{6}-\frac{\pi^2}{12}$, so $\lim_{t \to \infty} \beta = \frac{\pi}{2\sqrt{3}}$, and part (c) is proved.
\end{proof}

\begin{proposition}\label{P:B(t)andA(t)limits}
The function $B(t)$ in \eqref{E:B(t)andA_n(t)def} is strictly increasing, and we have the following limits for $B(t)$ and $A(t):= A_n(t) \left(1+e^{-\beta t}\right)^{\{t \sqrt{n}\}}$:
$$
\lim_{t \to \infty} B(t) = \frac{\pi}{\sqrt{3}} \qquad \text{and} \qquad \lim_{t \to \infty} A(t)= \frac{1}{4\sqrt[4]{3}}.
$$
\end{proposition}
\begin{proof}
We compute
\begin{align*}
B'(t) &= 2\beta'(t)- \frac{te^{-\beta t}}{1+e^{-\beta t}}\left(\beta'(t)t+\beta \right)+\log\left(1+e^{-\beta t}\right) \\
&=\beta'(t)\left(2-\frac{t^2e^{-\beta t}}{1+e^{-\beta t}}\right)-\frac{\beta te^{-\beta t}}{1+e^{-\beta t}}+ \log\left(1+e^{-\beta t}\right) \\ &=\log\left(1+e^{-\beta t}\right).
\end{align*}
Thus, $B(t)$ is a strictly increasing function, and using \eqref{E:betalimits} one easily sees that \newline $\lim_{t\to \infty} B(t)=\frac{\pi}{\sqrt{3}}$.

Finally, we can rewrite $A(t)$ using $\beta'(t)$ found in \eqref{E:betaprimedef}, and get
$$
A(t)= \frac{1}{2}\sqrt{\frac{\beta  \left(1+e^{-\beta t}\right)}{\pi\left(2-\frac{t^2}{1+e^{\beta t}}\right)}}.
$$
Using $\eqref{E:betalimits}$ and the boundedness of $\beta$, we have $\lim_{t \to \infty} A(t)= \frac{1}{4\sqrt[4]{3}}$.
\end{proof}

\section{Proofs of Propositions \ref{P:logDtasymp}, \ref{P:expvar} and \ref{P:Probasymp}}\label{S:Proofs}

Recall that $x$ depends on $n$ and $\beta$ as $x=e^{-\frac{\beta}{\sqrt{n}}}.$  In the proofs of Propositions \ref{P:logDtasymp} and \ref{P:Probasymp}, we will need to separate the cases $x>1$, $x<1$ and $x=1$, which after Proposition \ref{P:beta} correspond to $\sqrt{2}<t<2$, $t>2$ and $t=2$, respectively.  With this in mind, we define 
\begin{equation}\label{E:gammadef}
\gamma=\gamma(t):= -\beta(t), \qquad \text{for $\sqrt{2} < t< 2$,}
\end{equation}
so that $\gamma >0$ and $x^{-1}=e^{-\frac{\gamma}{\sqrt{n}}}<1$.

It is also necessary to account for the fact that $t\sqrt{n}$ is not always an integer.    Thus, we define
\begin{equation}\label{E:t_ndef}
t_n:= \frac{\floor{t\sqrt{n}}}{\sqrt{n}}= t-\frac{\{t\sqrt{n}\}}{\sqrt{n}},
\end{equation}
so that $t_n\sqrt{n} \in \mathbb{N}$, and a sum over $k \leq t\sqrt{n}$ is really a sum from $k=1$ to $t_n\sqrt{n}$.  Also, we may replace any differentiable function $f(t_n)$ with $f(t)+o(1)$.  We will often do this below when $f(t_n)$ is part of the constant term.

\begin{proof}[Proof of Proposition \ref{P:logDtasymp}]
Case 1: $t>2$.  The first iteration of Euler-Maclaurin summation (\cite{MV}, Appendix B) picks off the claimed main term and constant term.  Using \newline ${\rm Li}_2\left(-e^{-z}\right)=\int_{\infty}^z \log\left(1+e^{-w}\right)dw$ (substitute $w \to -e^{-w}$ in \eqref{E:Li2def}), we have
\begin{align}
&\log \D_{t,n}\left(x \right)
\nonumber \\ &= \sum_{k=1}^{t_n\sqrt{n}} \log \left(1+e^{\frac{-\beta k}{\sqrt{n}}}\right) \nonumber  \\
&= \int_1^{t\sqrt{n}} \log \left(1+e^{\frac{-\beta u}{\sqrt{n}}}\right)du-\int_{t_n\sqrt{n}}^{t\sqrt{n}} \log \left(1+e^{\frac{-\beta u}{\sqrt{n}}}\right)du \nonumber  \\
&\qquad+ \frac{1}{2}\left(\log \left(1+e^{\frac{-\beta }{\sqrt{n}}}\right) + \log \left(1+e^{-\beta t_n}\right)  \right) -\int_1^{t_n\sqrt{n}} \frac{\frac{\beta }{\sqrt{n}}e^{-\frac{\beta }{\sqrt{n}}u}}{1+e^{-\frac{\beta }{\sqrt{n}}u}}\left(\{u\}-\frac{1}{2}\right)du \nonumber  \\
&= \frac{\sqrt{n}}{\beta }\int_{\frac{\beta }{\sqrt{n}}}^{\beta t} \log \left(1+e^{-v}\right)dv - \frac{\sqrt{n}}{\beta}\int_{\beta t_n}^{\beta t} \log\left(1+e^{-v}\right)dv + \frac{1}{2}\log \left(1+e^{\frac{-\beta }{\sqrt{n}}}\right) \nonumber   \\
&\qquad + \frac{1}{2}\log \left(1+e^{-\beta t}\right)   +o(1)-\int_{\frac{\beta }{\sqrt{n}}}^{\beta t} \frac{e^{-v}}{1+e^{-v}}\left(\left\{\frac{\sqrt{n}v}{\beta }\right\}-\frac{1}{2}\right)dv \nonumber  \\
&= \frac{\sqrt{n}}{\beta }\left(\text{Li}_2\left(-e^{-\beta t}\right)-\text{Li}_2\left(-e^{-\frac{\beta }{\sqrt{n}}}\right) \right) -\{t\sqrt{n}\}\log\left(1+e^{-\beta t}\right) + \frac{1}{2}\log \left(1+e^{\frac{-\beta }{\sqrt{n}}}\right) \nonumber \\
&\qquad + \frac{1}{2}\log \left(1+e^{-\beta t}\right) +o(1) -\int_{\frac{\beta }{\sqrt{n}}}^{\beta t} \frac{e^{-v}}{1+e^{-v}}\left(\left\{\frac{\sqrt{n}v}{\beta}\right\}-\frac{1}{2}\right)dv \label{E:above}.
\end{align}
The latter integral is $o(1)$ because it is the product of an $L^1$ function and a bounded oscillating function\textemdash as in the proof of the Riemann-Lebesgue Lemma, we prove this first when $\frac{e^{-v}}{1+e^{-v}}$ is replaced by a step function, then we approximate $\frac{e^{-v}}{1+e^{-v}}$ in $L^1$ by step functions.  For the rest of the expression, we apply the following identity for the dilogarithm (\cite{AS}, 27.7.6)
\begin{equation}\label{E:Li2id}
\text{Li}_2(-x)=-\frac{\pi^2}{12} + \text{Li}_2\left(1-x\right)- \frac{1}{2}\text{Li}_2\left(1-x^2\right)- \log x \cdot \log(1+x).
\end{equation}
  Thus, recalling \eqref{E:betadef2}, we obtain the following from \eqref{E:above}
\begin{align}
&\frac{\sqrt{n}}{\beta }\left(\beta^2+t\beta \log\left(1+e^{-\beta t}\right) - \text{Li}_2\left(1-e^{-\frac{\beta }{\sqrt{n}}}\right) + \frac{1}{2}\text{Li}_2\left(1-e^{-\frac{2\beta }{\sqrt{n}}}\right) - \frac{\beta }{\sqrt{n}}\log\left(1+e^{-\frac{\beta }{\sqrt{n}}}\right) \right) \nonumber \\
&\qquad + \frac{1}{2}\log \left(1+e^{\frac{-\beta }{\sqrt{n}}}\right) + \frac{1}{2}\log \left(1+e^{-\beta t}\right) -\{t\sqrt{n}\}\log\left(1+e^{-\beta t}\right) + o(1). \label{E:justabove}
\end{align}
From the series representation for the dilogarithm, it follows that
$$
\frac{1}{x}{\rm Li_2}\left(1-e^{-x}\right)= \frac{1-e^{-x}}{x} + \frac{1}{x}\sum_{n \geq 2} \frac{\left(1-e^{-x}\right)^n}{n^2} = 1 + o(1), \qquad \text{as $x \to 0^+$}.
$$
Hence, \eqref{E:justabove} is
\begin{align*}
&\sqrt{n}\left(\beta +t\log\left(1+e^{-\beta t}\right)\right) - 1+1-\log(2) + \frac{1}{2}\log(2) + \frac{1}{2}\log\left(1+e^{-\beta t}\right) \\ &\qquad -\{t\sqrt{n}\}\log\left(1+e^{-\beta t}\right)+ o(1) \\
&= \sqrt{n}\left(\beta +t\log\left(1+e^{-\beta t}\right)\right) + \log\left(\sqrt{\frac{1+e^{-\beta t}}{2}}\right) -\{t\sqrt{n}\}\log\left(1+e^{-\beta t}\right)   + o(1).
\end{align*}

Combining this with $\log \left(x^{-n}\right)= \beta\sqrt{n}$ proves Proposition \ref{P:logDtasymp} for $t>2$.

\vspace{5mm}

Case 2: $\sqrt{2}<t<2$.  We first write
\begin{align}
\log \left(x^{-n}\D_{t,n}(x)\right)  
&= -\gamma\sqrt{n} +\sum_{k =1}^{t_n\sqrt{n}} \log\left(1+e^{\frac{\gamma k}{\sqrt{n}}}\right) \nonumber\\
&= -\gamma\sqrt{n} +\sum_{k =1}^{t_n\sqrt{n}} \left( \frac{\gamma k}{\sqrt{n}} +\log\left(1+e^{-\frac{\gamma k}{\sqrt{n}}}\right)\right) \nonumber \\
&= -\gamma\sqrt{n} + \frac{\gamma t_n(t_n\sqrt{n}+1)}{2}+\sum_{k =1}^{t_n\sqrt{n}}\log\left(1+e^{-\frac{\gamma k}{\sqrt{n}}}\right) \nonumber\\
&= \gamma\sqrt{n}\left(\frac{t_n^2}{2}-1\right)+ \frac{\gamma t_n}{2}+\sum_{k =1}^{t_n\sqrt{n}}\log\left(1+e^{-\frac{\gamma k}{\sqrt{n}}}\right) \nonumber\\
&= \gamma\sqrt{n}\left(\frac{t^2}{2}-1\right)- \gamma t \{t\sqrt{n}\}+ \frac{\gamma t}{2}+\sum_{k =1}^{t_n\sqrt{n}}\log\left(1+e^{-\frac{\gamma k}{\sqrt{n}}}\right) + o(1). \label{E:logDtgamma1}
\end{align}
We then analyze the sum with Euler-Maclaurin summation as before
\begin{align}
&\sum_{k=1}^{t_n\sqrt{n}} \log\left(1+ e^{-\frac{\gamma k}{\sqrt{n}}}\right) \nonumber \\
&= \frac{\sqrt{n}}{\gamma}\left({\rm Li}_2\left(-e^{-\gamma t}\right) -{\rm Li}_2\left(-e^{-\frac{\gamma}{\sqrt{n}}}\right)\right) + \frac{1}{2}\log\left(1+e^{-\frac{\gamma}{\sqrt{n}}}\right) + \frac{1}{2}\log\left(1+e^{-\gamma t}\right) \nonumber \\ & \qquad -\{t\sqrt{n}\}\log\left(1+e^{-\gamma t}\right)+ o(1). \label{E:above2}
\end{align}
This time we rewrite \eqref{E:betadef}, the integral definition for $\beta=-\gamma$, to get
\begin{equation}\label{E:gammadef2}
\gamma^2= \int_0^{\gamma t} \frac{u}{1+e^{-u}}du=-\frac{\pi^2}{12} + \frac{\gamma^2t^2}{2}+ \gamma t \log\left(1+e^{-\gamma t}\right)- {\rm Li}_2\left(-e^{-\gamma t}\right).
\end{equation}
We then apply this and the dilogarithm identity \eqref{E:Li2id} to get the following from \eqref{E:above2}, in a manner similar to Case 1 
\begin{align*}
&\frac{\sqrt{n}}{\gamma}\left(\gamma^2\left(\frac{t^2}{2}-1\right) + \gamma t \log\left(1+e^{-\gamma t}\right) - {\rm Li}_2\left(1-e^{-\frac{\gamma}{\sqrt{n}}}\right)+ \frac{1}{2}{\rm Li}_2\left(1-e^{-2\frac{\gamma}{\sqrt{n}}}\right) \right. \\ & \qquad \left.- \frac{\gamma}{\sqrt{n}}\log\left(1+e^{-\frac{\gamma}{\sqrt{n}}}\right)\right)  +\frac{1}{2}\log\left(1+e^{-\frac{\gamma}{\sqrt{n}}}\right) + \frac{1}{2}\log\left(1+e^{-\gamma t}\right)  -\{t\sqrt{n}\}\log(1+e^{-\gamma t}) \\ &\qquad + o(1)  \\
&= \sqrt{n} \left(\gamma \left(\frac{t^2}{2}-1\right) + t\log\left(1+e^{-\gamma t}\right)\right) + \log \sqrt{\frac{1+e^{-\gamma t}}{2}} -\{t\sqrt{n}\}\log\left(1+e^{-\gamma t}\right)+ o(1).
\end{align*}
Combining with \eqref{E:logDtgamma1}, we have the following expression for $\log\left(x^{-n}\D_{t,n}(x)\right)$
\begin{align*}
&\sqrt{n} \left(\gamma \left(t^2-2\right) + t\log\left(1+e^{-\gamma t}\right)\right) + \log \sqrt{\frac{1+e^{-\gamma t}}{2}} +\frac{\gamma t}{2}-\{t\sqrt{n}\}\left(\gamma t +\log\left(1+e^{-\gamma t}\right)\right)+ o(1) \\
&=\sqrt{n} \left(-2\gamma + t\log\left(1+e^{\gamma t}\right)\right) + \log \sqrt{\frac{1+e^{\gamma t}}{2}} -\{t\sqrt{n}\}\log\left(1+e^{\gamma t}\right)+ o(1).
\end{align*}
Replacing $\gamma$ with $-\beta$ completes the proof for $\sqrt{2}<t<2.$

\vspace{5mm}

Case 3: $t=2$.  Here, we have $\beta=0$ and $x=1$, and so
$$
x^{-n}\D_{2,n}(x)= 2^{t_n\sqrt{n}}= 2^{2\sqrt{n}-\{2\sqrt{n}\}}= e^{2\log 2 \sqrt{n} - \log 2 \{2\sqrt{n}\}},
$$
as required.
\end{proof}

\begin{proof}[Proof of Proposition \ref{P:expvar}]
The proof for $t=2$ (and so $x=1$) is straightforward.  Now let $t \neq 2$.  We need only recognize the Riemann Sums
\begin{align*}
{\rm E}_{x,n}(N)
&= \sum_{k\leq t\sqrt{n}} k \frac{e^{-\frac{\beta k}{\sqrt{n}}}}{1+e^{-\frac{\beta k}{\sqrt{n}}}} = n\sum_{k\leq t\sqrt{n}} \frac{k}{\sqrt{n}} \frac{e^{-\frac{\beta k}{\sqrt{n}}}}{1+e^{-\frac{\beta k}{\sqrt{n}}}} \cdot \frac{1}{\sqrt{n}} \\
&= n\left(\int_0^t \frac{ue^{-\beta u}}{1+e^{-\beta u}}du + O\left(\frac{1}{\sqrt{n}}\right)\right) = n + O(\sqrt{n}),
\end{align*}
by \eqref{E:betadef}.  We calculate the variance similarly, using integration by parts to evaluate the integral.  We also use the fact that $\frac{e^{-u}}{(1+e^{-u})^2}=\frac{e^{u}}{(1+e^{u})^2}.$  Thus,
\begin{align*}
{\rm Var}_{x,n}(N)
&= \sum_{k\leq t\sqrt{n}} k^2 \frac{e^{\frac{-\beta k}{\sqrt{n}}}}{\left(1+e^{-\frac{\beta k}{\sqrt{n}}}\right)^2} \\
&=n^{\frac{3}{2}}\sum_{k\leq t\sqrt{n}} \left(\frac{k}{\sqrt{n}}\right)^2 \frac{e^{\frac{-\beta k}{\sqrt{n}}}}{\left(1+e^{-\frac{\beta k}{\sqrt{n}}}\right)^2} \cdot \frac{1}{\sqrt{n}} \\
&=n^{\frac{3}{2}}\left(\int_0^t \frac{u^2e^{-\beta u}}{\left(1+e^{-\beta u}\right)^2}du +O\left(\frac{1}{\sqrt{n}}\right) \right) \\
&=\frac{n^{\frac{3}{2}}}{\beta^3}\int_0^{\beta t} \frac{u^2e^{ u}}{\left(1+e^{ u}\right)^2}du +O\left(n\right) \\
&= \frac{n^{\frac{3}{2}}}{\beta^3}\left(-\frac{u^2}{1+e^u} \Bigg|_0^{\beta t} + 2 \int_0^{\beta t} \frac{u}{1+e^u} du \right) + O(n) \\
& =  \frac{n^{\frac{3}{2}}}{\beta^3}\left(-\frac{\beta^2 t^2}{1+e^{\beta t}} + 2\beta^2\right) + O(n).
\end{align*}
by \eqref{E:beta^2int}.  Combining and recalling \eqref{E:betaprimedef} finishes the proof.
\end{proof}

The proof of Proposition \ref{P:Probasymp} is the most technical and will require the following two lemmas.
\begin{lemma}\label{L:f_xbound}
Let
\begin{equation}\label{E:f_xdef}
f_x(s):= \log\left(\frac{1+e^{is}x}{1+x}\right) - is\frac{x}{1+x} + \frac{s^2}{2} \frac{x}{(1+x)^2}.
\end{equation}
There exists a constant $c>0$ such that for any $x\in (0,1)$ and any $s\in \R$, we have $$
|f_x(s)|\leq c\frac{x|s|^3}{(1-x)^3}.
$$
\end{lemma}

\begin{lemma}\label{L:rothszekeres}
Let $\epsilon \in \left(0, \frac{1}{2}\right]$ and let $\|\alpha\|$ denote the distance between $\alpha$ and the nearest integer.  Then
$$
\inf_{\frac{\epsilon}{n} \leq \alpha \leq \frac{1}{2}} \sum_{k \leq n} \left\|k\alpha\right\|^2 \gg n.
$$
\end{lemma}
\noindent We append the proofs of these lemmas to Section \ref{S:lemmaproofs}.  The proof of Lemma \ref{L:f_xbound} is similar to the proof of Lemma 1 in \cite{R}.  Roth and Szekeres \cite{RS} proved Lemma \ref{L:rothszekeres} for $\frac{1}{2n} \leq x \leq \frac{1}{2}$ when $\{k\}$ is replaced by a much more general sequence, but with a weaker lower bound.

\begin{proof}[Proof of Proposition \ref{P:Probasymp}]
To determine the asymptotic behavior of $P_{x,n}(N=n)$, we will apply Fourier inversion to the characteristic function for $N$,
\begin{align*}
\phi_{x,n}(s):= {\rm E}_{x,n}(e^{isN})= \sum_{k \geq 0} P_{x,n}(N=k)e^{isk}&=\frac{1}{\D_{t,n}(x)}\sum_{k \geq 0} \left({\rm Coeff} [x^k] \D_{t,n}(x)\right) x^k e^{isk} \\ &= \frac{\D_{t,n}(xe^{is})}{\D_{t,n}(x)}.
\end{align*}
Note that $\phi_{x,n}$ depends on $n$, although we refrain from notating this.  We have
\begin{equation}\label{normalint}
P_{x,n}(N=n) = \frac{1}{2\pi}\int_{-\pi}^{\pi} \phi_{x,n}(s)e^{-ins}ds  =\frac{1}{2\pi\sigma_n}\int_{-\pi\sigma_n}^{\pi\sigma_n} \phi_{x,n}\left(\frac{u}{\sigma_n}\right) e^{-i\frac{nu}{\sigma_n}}du.
\end{equation}
We split this integral as
\begin{equation}\label{E:intbreak}
\int_{-\pi\sigma_n}^{\pi\sigma_n} = \int_{|u|\leq \frac{\sigma_n}{\sqrt{n}}v_0} + \int_{ \frac{\sigma_n}{\sqrt{n}}v_0 \leq |u| \leq \pi\sigma_n} ,
\end{equation}
where $v_0$ is a sufficiently small constant, depending on $t$ and chosen below.  Note that $\sigma_n \asymp n^{\frac{3}{4}}$, so $\frac{\sigma_n}{\sqrt{n}} \to \infty$.  We show that the integral on the right in \eqref{E:intbreak} tends to 0, while for the left integral we show that 
\begin{equation}\label{phitobellcurve}
\lim_{n \to \infty} \phi_{x,n}\left(\frac{u}{\sigma_n}\right) e^{-i\frac{nu}{\sigma_n}} = e^{-\frac{u^2}{2}},
\end{equation}
pointwise in $u$.  We then show that for some $A'>0$, the integrand $\phi_{x,n}\left(\frac{u}{\sigma_n}\right)$ is dominated by $e^{-A'u^2} \in L^1(\mathbb{R})$.  Thus, applying the Dominated Convergence Theorem, 
\begin{equation}\label{E:limitsqrt2pi}
\lim_{n \to \infty} \int_{|u|\leq \frac{\sigma_n}{\sqrt{n}}v_0}  \phi_{x,n}\left(\frac{u}{\sigma_n}\right)e^{-i\frac{nu}{\sigma_n}}du = \int_{\mathbb{R}} e^{-\frac{u^2}{2}} = \sqrt{2\pi},
\end{equation}
which, when combined with \eqref{normalint}, proves that $P_{x,n}(N=n) \sim \frac{1}{\sqrt{2\pi} \sigma_n}$.  A similar application of the Dominated Convergence Theorem also implies \eqref{E:Nnormal}, since the characteristic function of $\frac{N-n}{\sigma_n}$ is
$$
{\rm E}_{x,n}\left(e^{iu \frac{N-n}{\sigma_n}}\right)={\rm E}_{x,n}\left(e^{i \frac{u}{\sigma_n}N}\right) e^{-iu\frac{n}{\sigma_n}}= \phi_{x,n}\left(\frac{u}{\sigma_n}\right)e^{-iu\frac{n}{\sigma_n}}.
$$

To carry out this plan, we separate the cases $t>2$, $\sqrt{2}<t<2$ and $t=2$.

\vspace{5mm}
Case 1: $t>2$.  Recalling the expectation and variance in \eqref{E:expvardef}, Proposition \ref{P:expvar} implies
\begin{align}
 &\log \left(\phi_{x,n}\left(\frac{u}{\sigma_n}\right) e^{-i\frac{nu}{\sigma_n}}\right)\\  &= \log\left(\D_{t,n}(xe^{i\frac{u}{\sigma_n}})\right)-\log\left(\D_{t,n}(x)\right) -i\frac{nu}{\sigma_n} \nonumber \\
&=\sum_{k \leq t\sqrt{n}} \log\left(\frac{1+x^ke^{i\frac{ku}{\sigma_n}}}{1+x^k}\right)- i\frac{nu}{\sigma_n} \nonumber \\
&= i\frac{u}{\sigma_n}\left(\sum_{k \leq t\sqrt{n}} \frac{kx^k}{1+x^k} - n \right) - \frac{u^2}{2\sigma_n^2}\left(\sum_{k \leq t\sqrt{n}} \frac{k^2x^k}{\left(1+x^k\right)^2}\right) + \sum_{k \leq t\sqrt{n}} f_{x^k}\left(\frac{ku}{\sigma_n}\right) \nonumber \\
&= i\frac{u}{\sigma_n}({\rm E}_{x,n}(N)-n) - \frac{u^2}{2} +\sum_{k \leq t\sqrt{n}} f_{x^k}\left(\frac{ku}{\sigma_n}\right) \nonumber \\
&= iuO\left(n^{-\frac{1}{4}}\right) - \frac{u^2}{2}+\sum_{k \leq t\sqrt{n}} f_{x^k}\left(\frac{ku}{\sigma_n}\right), \label{Proberror}
\end{align}
where $f_{x^k}$ is as in Lemma \ref{L:f_xbound}.  Using Proposition \ref{P:expvar} and Lemma \ref{L:f_xbound}, we have
\begin{align*}
\left|\sum_{k \leq t\sqrt{n}} f_{x^k}\left(\frac{ku}{\sigma_n}\right) \right| &\leq \frac{cu^3}{\sigma_n^3} \sum_{k \leq t\sqrt{n}} k^3 \frac{x^k}{(1-x^k)^3} \\
&= \frac{cu^3n^2}{\sigma_n^3}\sum_{k \leq t\sqrt{n}} \left(\frac{k}{\sqrt{n}}\right)^3 \frac{e^{-\frac{\beta k}{\sqrt{n}}}}{\left(1-e^{-\frac{\beta k}{\sqrt{n}}}\right)^3} \cdot \frac{1}{\sqrt{n}} \\
&= \frac{cu^3n^2}{\sigma_n^3}\left(\int_0^t \frac{v^3 e^{-\beta v}}{\left(1-e^{-\beta v}\right)^3} dv + O\left(\frac{1}{\sqrt{n}}\right)\right) \\
&= u^3 O\left(n^{-\frac{1}{4}}\right),
\end{align*}
 since the integral converges.  This proves \eqref{phitobellcurve}.

Next, we find a dominating function in the range $|u| \leq \frac{\sigma_n}{\sqrt{n}}v_0$.  Here, we will set $v:= \frac{\sqrt{n}}{\sigma_n}u$, so $|v|\leq v_0$.  Recognizing Riemann sums, the following holds for such $v$ uniformly

\begin{align}
\log \phi_{x,n}\left(\frac{v}{\sqrt{n}}\right) &= \sum_{k \leq t\sqrt{n}} \left(\log \left(1+e^{-\frac{\beta }{\sqrt{n}}k + i\frac{vk}{\sqrt{n}}}\right)-\log \left(1+e^{-\frac{\beta }{\sqrt{n}}k}\right) \right) \nonumber\\
&=\sqrt{n} \int_0^{t} \left(\log\left(1+e^{-\beta w+ivw}\right)-\log\left(1+e^{-\beta w}\right)\right) dw + o(\sqrt{n}) \nonumber\\
&=\frac{\sqrt{n}}{\beta -iv}\left(\frac{\pi^2}{12} + \text{Li}_2\left(-e^{-\beta t+ivt}\right)\right) - \frac{\sqrt{n}}{\beta }\left(\frac{\pi^2}{12} + \text{Li}_2\left(-e^{-\beta t}\right)\right) + o(\sqrt{n}). \label{Liint}
\end{align}
The Taylor series for Li$_2(z)$ about  $z=-e^{-\beta t}$ in the variable $z$ is
$$
\text{Li}_2\left(-e^{-\beta t}\right) + \log\left(1+e^{-\beta t}\right)\left(ze^{\beta t}+1\right)- \frac{1}{2} \left(\frac{e^{-\beta t}}{1+e^{-\beta t}} - \log\left(1+e^{-\beta t}\right)\right) \left(ze^{\beta t}+1\right)^2 $$ $$+ O\left(\left(ze^{\beta t}+1\right)^3\right).
$$
Substituting $z=-e^{-\beta t+ivt}$, we obtain
\begin{align}
& \text{Li}_2\left(-e^{-\beta t+ivt}\right)  \nonumber \\
&= \text{Li}_2\left(-e^{-\beta t}\right)+\log\left(1+e^{-\beta t}\right)\left(1-e^{ivt}\right)\nonumber \\
&\qquad - \frac{1}{2} \left(\frac{e^{-\beta t}}{1+e^{-\beta t}} - \log\left(1+e^{-\beta t}\right)\right) \left(1-e^{ivt}\right)^2  + O(v^3) \nonumber \\
&=\text{Li}_2\left(-e^{-\beta t}\right)-i t\log\left(1+e^{-\beta t}\right)v \nonumber \\  & \qquad + \left( \frac{t^2\log\left(1+e^{-\beta t}\right)}{2} + \frac{t^2}{2} \cdot \frac{e^{-\beta t}}{1+e^{-\beta t}} -\frac{t^2\log\left(1+e^{-\beta t}\right)}{2} \right)v^2 + O(v^3) \nonumber \\
&=\text{Li}_2\left(-e^{-\beta t}\right) -it \log\left(1+e^{-\beta t}\right)v +  \frac{1}{2} \cdot \frac{t^2}{1+e^{\beta t}} v^2 + O(v^3) \label{Litay2}.
\end{align}

Also, note that
\begin{equation}\label{taybetais}
\frac{1}{\beta-iv}= \frac{1}{\beta}+ \frac{i}{\beta^2}v - \frac{1}{\beta^3}v^2 + O(v^3).
\end{equation}
Thus, from \eqref{Litay2} and \eqref{taybetais}, we choose $v_0$ small enough so that the dominating term for the real part of \eqref{Liint} is
\begin{align*}
&\sqrt{n}v^2\left(\frac{1}{\beta } \cdot \frac{1}{2} \cdot \frac{t^2}{1+e^{\beta t}} + \frac{t \log\left(1+e^{-\beta t}\right)}{\beta^2} - \frac{1}{\beta^3}\text{Li}_2\left(-e^{-\beta t}\right) - \frac{1}{\beta^3} \cdot \frac{\pi^2}{12} \right) \\
&=\sqrt{n}\frac{v^2}{\beta} \left(\frac{1}{2} \cdot \frac{t^2}{1+e^{\beta t}} - 1 \right),
\end{align*}
where we used the dilogarithm identity \eqref{E:Li2id} with the alternate definition of $\beta$ given in \eqref{E:betadef2}.  By \eqref{betapos}, this is $-A\sqrt{n}v^2$ for some $A>0$.  Hence, for some $A>0$,
$$
\left|\phi_{x,n}\left(\frac{v}{\sqrt{n}}\right)\right| \ll e^{-A\sqrt{n}v^2} \qquad \text{for $|v| \leq v_0$.}
$$

This implies $\left|\phi_{x,n}\left(\frac{u}{\sigma_n}\right)\right| \ll e^{-A'u^2}$ for some $A'$ in the required range.  Thus, \eqref{E:limitsqrt2pi} is proved.

For the remaining range, $\frac{\sigma_n}{\sqrt{n}}v_0 \leq |u| \leq \pi \sigma_n$, we will use the substitution $w:= \frac{u}{\sigma_n}$ and bound $\phi_{x,n}\left(w\right)$ for  $\frac{v_0}{\sqrt{n}} \leq \left| w\right| \leq \pi$.  Following the analysis of Roth and Szekeres (\cite{RS}, p. 253), we write
$$
\left|\frac{1+x^ke^{iwk}}{1+x^k}\right|^2= \frac{1}{(1+x^k)^2}\left(1+2x^k\cos(wk) + x^{2k}\right) = 1-\frac{2x^k(1-\cos(wk))}{(1+x^k)^2}.
$$
Note that the expression on the far left is positive almost everywhere; therefore, \newline $0<\frac{2x^k(1-\cos(wk))}{(1+x^k)^2}<1$ almost everywhere.  Thus, it is safe to expand the logarithm as follows
\begin{align*}
\log\left|\phi_{x,n}(w)\right| &=\frac{1}{2}\sum_{k \leq t\sqrt{n}} \log\left(1-\frac{2x^k(1-\cos(wk))}{(1+x^k)^2} \right) \leq -\frac{1}{2}\sum_{k \leq t\sqrt{n}} \left(1-\cos(wk)\right)\frac{2x^k}{(1+x^k)^2} \\
&\ll -\sum_{k \leq t\sqrt{n}} (1-\cos(wk))
\leq -\sum_{k \leq t\sqrt{n}} \left\|\frac{wk}{2\pi} \right\|^2.
\end{align*}
Since $\frac{v_0}{\sqrt{n}} \leq |w| \leq \pi,$ the latter is $\ll - \sqrt{n}$ by Lemma \ref{L:rothszekeres}, taking $\epsilon \leq t v_0$, so that $ \frac{\epsilon}{t\sqrt{n}}\leq\frac{v_0}{\sqrt{n}} .$  This implies that the right integral in \eqref{E:intbreak} tends to 0, so Proposition \ref{P:Probasymp} is proved for $t>2$.

\vspace{5mm}

Case 2: $\sqrt{2} < t < 2$.  Below, we use the fact that $\frac{x^k}{(1+x^k)^2}= \frac{x^{-k}}{(1+x^{-k})^2}.$  Thus,
\begin{align*}
&\log \left(\phi_{x,n}\left(\frac{u}{\sigma_n}\right)e^{-i\frac{un}{\sigma_n}}\right) \\
&=-i\frac{un}{\sigma_n} + \sum_{k\leq t\sqrt{n}} \log\left(\frac{1+x^{-k}e^{-ik\frac{u}{\sigma_n}}}{1+x^{-k}} e^{ik\frac{u}{\sigma_n}}\right) \\
&=i\frac{u}{\sigma_n}\left(\sum_{k \leq t\sqrt{n}} \frac{kx^k}{1+x^k} -n \right) - \frac{u^2}{2\sigma_n^2}\sum_{k \leq t\sqrt{n}} \frac{k^2x^k}{(1+x^k)^2} + i\frac{u}{\sigma_n}\frac{t_n\sqrt{n}(t_n\sqrt{n}+1)}{2} + \\ & \qquad +\sum_{k \leq t\sqrt{n}} \left(\log\left(\frac{1+x^{-k}e^{-ik\frac{u}{\sigma_n}}}{1+x^{-k}}\right) - i\frac{u}{\sigma_n}\frac{kx^k}{1+x^k} + \frac{u^2}{2\sigma_n^2}\frac{k^2x^k}{(1+x^k)^2} \right) \\
&= o(1) - \frac{u^2}{2} + i\frac{u}{\sigma_n}\frac{t_n\sqrt{n}(t_n\sqrt{n}+1)}{2} + \sum_{k \leq t \sqrt{n}} f_{x^{-k}}\left(-k\frac{u}{\sigma_n}\right) \\ & \qquad -i \frac{u}{\sigma_n} \sum_{k \leq t\sqrt{n}} k\underbrace{\left(\frac{x^k}{1+x^k} + \frac{x^{-k}}{1+x^{-k}}\right)}_{=1} \\
&=  o(1) - \frac{u^2}{2}  + \sum_{k \leq t \sqrt{n}} f_{x^{-k}}\left(-k\frac{u}{\sigma_n}\right) \\
&= -\frac{u^2}{2} + o(1),
\end{align*}
where Lemma \ref{L:f_xbound} was used as before to show that the sum is $o(1)$.  This proves \eqref{phitobellcurve}.

To find a dominating function in the range $|u| \leq \frac{\sigma_n}{\sqrt{n}}v_0$, we once again set $v:= \frac{\sqrt{n}}{\sigma_n}u$, and write
$$
\left|\phi_{x,n}\left(\frac{v}{\sqrt{n}}\right) \right|= \left|e^{i\frac{t_n(t_n\sqrt{n}+1)}{2}v} \prod_{k \leq t\sqrt{n}} \frac{1+x^{-k}e^{-ik\frac{v}{\sqrt{n}}}}{1+x^{-k}}\right|=\left|\phi_{x^{-1},n} \left(\frac{-v}{\sqrt{n}}\right)\right|.
$$
Thus, we may perform an analysis similar to Case 1 with $\beta \to \gamma$ and conclude that the dominating part of Re$\left( \log \phi_{x,n}\left(\frac{v}{\sqrt{n}}\right) \right)$ is
$$
\sqrt{n}(-v)^2\left(\frac{1}{\gamma } \cdot \frac{1}{2} \cdot \frac{t^2}{1+e^{\gamma t}} + \frac{t \log\left(1+e^{-\gamma t}\right)}{\gamma^2} - \frac{1}{\gamma^3}\text{Li}_2\left(-e^{-\gamma t}\right) - \frac{1}{\gamma^3} \cdot \frac{\pi^2}{12} \right). 
$$
We now apply the identity \eqref{E:Li2id} for the dilogarithm with \eqref{E:gammadef2} to get
$$
\sqrt{n} \frac{v^2}{\gamma}\left(\frac{t^2}{2\left(1+e^{\gamma t}\right)} + 1 - \frac{t^2}{2}\right) 
=\sqrt{n}\frac{v^2}{\gamma}\left(\frac{2\left(1+e^{-\gamma t}\right)-t^2}{2\left(1+e^{-\gamma t}\right)}\right) 
= \sqrt{n} v^2 \left(\frac{-t}{\beta'(t) \cdot 2\left(1+e^{\beta t}\right)}\right)
,$$
which is negative by Proposition \ref{P:beta}.  Thus, as in Case 1, $\left|\phi_{x,n}\left(\frac{u}{\sigma_n}\right)\right| \ll e^{-A'u^2}$ for some $A'$ in the required range, so \eqref{E:limitsqrt2pi} is proved.

As in Case 1, a similar application of Lemma \ref{L:rothszekeres} to $\phi_{x^{-1},n}(-w)$ shows that the right integral in \eqref{E:intbreak} tends to 0, so Proposition \ref{P:Probasymp} is proved for $\sqrt{2}<t<2$.

\vspace{5mm}

Case 3: $t=2$.  For fixed $u$ in the range $|u|\leq \frac{\sigma_n}{\sqrt{n}}v_0$, where $v_0$ will be specified below, we write
\begin{align*}
\phi_{1,n}\left(\frac{u}{\sigma_n}\right)e^{-i\frac{un}{\sigma_n}} &=\prod_{k \leq 2\sqrt{n}} \frac{1+e^{ik\frac{u}{\sigma_n}}}{2} \cdot e^{-i\frac{un}{\sigma_n}} \\ &= \prod_{k \leq 2\sqrt{n}} \frac{e^{-ik\frac{u}{2\sigma_n}}+e^{ik\frac{u}{2\sigma_n}}}{2}e^{ik\frac{u}{2\sigma_n}} \cdot e^{-i\frac{un}{\sigma_n}} \\ &= \prod_{k \leq 2 \sqrt{n}} \cos\left(k \frac{u}{2\sigma_n}\right) \cdot e^{i \frac{u}{\sigma_n} \left(\frac{t_n\sqrt{n}(t_n\sqrt{n}+1)}{4}-n\right)} \\ &=\prod_{k \leq 2 \sqrt{n}} \cos\left(k \frac{u}{2\sigma_n}\right) + o(1),
\end{align*}
since 
$$
\frac{t_n\sqrt{n}(t_n\sqrt{n}+1)}{4}= \frac{t_n^2 n}{4}+O(\sqrt{n})= n + O(\sqrt{n}).
$$
Note that, over the summation range, $k= O(\sqrt{n})$ uniformly, so $\frac{k}{\sigma_n}= O\left(\frac{1}{\sqrt[4]{n}}\right)$ uniformly.  Thus, the following holds for fixed $u$, where $v_0$ is chosen so that the logarithms below are defined,
\begin{align*}
\log\left(\phi_{1,n}\left(\frac{u}{\sigma_n}\right)e^{-i\frac{un}{\sigma_n}}\right) &= \sum_{k \leq 2\sqrt{n}} \log \left(\cos\left(k\frac{u}{2\sigma_n}\right)\right) + o(1) \\
&=\sum_{k \leq 2\sqrt{n}} \log \left(1-k^2\frac{u^2}{4\sigma_n^2}+ O\left(\frac{1}{n}\right)\right) + o(1) \\
&=\sum_{k \leq 2\sqrt{n}} \left(-k^2\frac{u^2}{4\sigma_n^2}+ O\left(\frac{1}{n}\right)\right) + o(1) \\
&=-\frac{u^2 \cdot t_n^3n^{\frac{3}{2}}}{4 \sigma_n^2 \cdot 3} + o(1) \\
&=-\frac{u^2}{2} +o(1),
\end{align*}
since $\sigma_n^2 \sim \frac{2}{3}n^{\frac{3}{2}}$ by Propositions \ref{P:expvar} and \ref{P:beta}.  This proves \eqref{phitobellcurve}.

To find a dominating function in the range $|u| \leq \frac{\sigma_n}{\sqrt{n}}v_0$, we write $\frac{v}{\sqrt{n}}:= \frac{u}{\sigma_n}$ once again, and we choose $v_0$ small so that the logarithms below are defined.  Thus,
\begin{align*}
{\rm Re}\left(\log \phi_{1,n}\left(\frac{v}{\sqrt{n}}\right)\right) &= \sum_{k \leq 2\sqrt{n}} \log \cos\left(v\frac{k}{2\sqrt{n}}\right) \\
&= 2\sqrt{n}\int_0^1 \log \cos\left(vw\right)dw + O(\{2\sqrt{n}\}) + O\left(\frac{1}{\sqrt{n}}\right) \\
&=\frac{2\sqrt{n}}{v}\int_0^v \log \cos (w) dw + O(1).
\end{align*}
It is not difficult to calculate the following Taylor series about $v=0$ (the knowledge that the function is even is helpful)
$$
\frac{1}{v}\int_0^v \log \cos (w) dw= -\frac{1}{6}v^2 + O(v^4).
$$
Thus, choosing $v_0$ small enough, we have ${\rm Re}\left(\log \phi_{1,n}\left(\frac{v}{\sqrt{n}}\right)\right) \ll -A\sqrt{n}v^2$ for some $A$, which implies $\left|\phi_{1,n}\left(\frac{u}{\sigma_n}\right)\right| \ll e^{-A'u^2}$ for some $A'$ in the required range for $u$, so \eqref{E:limitsqrt2pi} is proved.

Applying Lemma 2 as in Case 1, one can bound $\phi_{1,n}(w)$ for $w= \frac{u}{\sigma_n}$ in the required range, and show that the right integral in \eqref{E:intbreak} tends to 0.  This proves Proposition \ref{P:Probasymp} for $t=2$.
\end{proof}

\section{Bounding logarithmic series: proofs of Lemmas \ref{L:f_xbound} and \ref{L:rothszekeres}}\label{S:lemmaproofs}

\begin{proof}[Proof of Lemma \ref{L:f_xbound}] The proof is very similar to Lemma 1 in \cite{R}.
For $|s| \leq \frac{1-x}{2},$ we have
\begin{align}
\log\left(\frac{1+xe^{is}}{1+x}\right) = \sum_{j \geq 1} \frac{1}{j}\left((-x)^j-(-x)^je^{isj}\right) &=-\sum_{j \geq 1} \frac{(-x)^j}{j}\sum_{k\geq 1} \frac{(is)^kj^k}{k!} \nonumber \\& = -\sum_{k \geq 1} \frac{(is)^k}{k!} \sum_{j \geq 1} (-x)^jj^{k-1}, \label{swapsum}
\end{align}
where swapping the order of summation in \eqref{swapsum} is valid due to absolute convergence for \newline $|s| \leq \frac{1-x}{2}.$  Indeed,
$$
\left|\sum_{k \geq 1} \frac{(is)^k}{k!} \sum_{j \geq 1} (-x)^jj^{k-1}\right|\leq \sum_{k \geq 1} \frac{s^k}{k} \sum_{j \geq 1} x^j \frac{n(n+1)\cdots (j+k-2)}{(k-1)!} \leq \sum_{k \geq 1} \frac{x}{k} \left(\frac{s}{1-x}\right)^k < \infty.
$$

Note that the $k=1$ and $k=2$ terms in \eqref{swapsum} are, respectively,
$$
-is\sum_{j\geq 1} (-x)^j= is\frac{x}{1+x} \qquad \text{and} \qquad \frac{s^2}{2}\sum_{j \geq 1} (-x)^jj= -\frac{s^2}{2}\frac{x}{(1+x)^2}.
$$
Thus, by \eqref{swapsum}, we obtain
\begin{align*}
|f_x(s)| &\leq \sum_{k \geq 3} \frac{|s|^k}{k!}\sum_{j \geq 1} j^{k-1}x^j \leq \sum_{k \geq 3} \frac{x}{k} \left(\frac{|s|}{1-x}\right)^k \leq \frac{x|s|^3}{3(1-x)^3} \cdot \frac{1}{1-\frac{|s|}{1-x}} \leq \frac{2x|s|^3}{3(1-x)^3}.
\end{align*}

For $|s| \geq \frac{1-x}{2}$, we have
$$
\left|-i \frac{x}{1+x}s + \frac{1}{2} \frac{x}{(1+x)^2} s^2 \right| \leq \frac{x|s|^3}{(1-x)|s|^2} + \frac{x|s|^3}{(1-x)^2|s|} \leq (4+2) \frac{x|s|^3}{(1-x)^3},
$$
so it remains to prove that for $|s|\geq \frac{1-x}{2}$,
$$
\left|\log\left(\frac{1+xe^{is}}{1+x}\right)\right| \leq c' \frac{x|s|^3}{(1-x)^3},
$$
for some $c' > 0.$  For $|s| \geq \frac{1}{4}$, we have
\begin{align*}
\left|\log\left(\frac{1+xe^{is}}{1+x}\right)\right| \leq \sum_{m \geq 1} \frac{x^m}{m} \left|1-e^{ism}\right| \leq -2\log(1-x) \leq 2\frac{x}{1-x} \leq 2\cdot 4^3\frac{x|s|^3}{(1-x)^3}.
\end{align*}
Finally, for $\frac{1-x}{2} \leq |s| \leq \frac{1}{4}$ (which implies $x \geq \frac{1}{2}$ and $\frac{|s|}{1-x} \geq \frac{1}{2}$), we have
\begin{align*}
\left|\log\left(\frac{1+xe^{is}}{1+x}\right) \right|
&\leq \left|\log\left(1+\frac{x}{1+x}(e^{is}-1)\right) \right| \\
&\leq \left|\log\left(1+ie^{\frac{s}{2}}S\right)\right|,
\end{align*}
where $S:= \frac{x}{1+x} \cdot 2\sin\left(\frac{s}{2}\right)$ satisfies $$|S| \leq \frac{x|s|}{1+x} \leq \frac{x|s|}{1-x} \leq 4\frac{x|s|^3}{(1-x)^3}.$$
This implies $|S| \leq \frac{1}{4}$.  Thus, $$\left|\log\left(1+ie^{\frac{s}{2}}S\right)\right| \leq \frac{|S|}{1-|S|} \leq \frac{4}{3}|S| \leq \frac{4}{3}\cdot 4\frac{x|s|^3}{(1-x)^3},$$ and we are done.
\end{proof}

\begin{proof}[Proof of Lemma \ref{L:rothszekeres}]

Let $f_n(\alpha):= \sum_{k\leq n} ||k\alpha||^2.$  We first prove that
\begin{equation}\label{E:inffn}
\inf_{\frac{1}{2n} \leq \alpha \leq \frac{1}{2}} f_n(\alpha) \gg n.
\end{equation}

The extension of \eqref{E:inffn} to the range $\left[\frac{\epsilon}{n}, \frac{1}{2}\right]$ for $\epsilon \in \left(0, \frac{1}{2}\right]$ follows from $$\inf_{\frac{\epsilon}{n} \leq \alpha \leq \frac{1}{2n}} f_n(\alpha) = \sum_{k \leq n} k^2 \frac{\epsilon^2}{n^2} \gg n.$$

Now, each summand $||k\alpha||^2$ is a piecewise parabola of the form $k^2\left(\alpha- \frac{\ell_k}{k}\right)^2$; therefore, $f_n(\alpha)$ may also be viewed as piecewise of the form $$\sum_{k \leq n} k^2\left(\alpha- \frac{\ell_k}{k}\right)^2,$$ for some integers $\ell_k$. Thus, we see by taking the derivative of $f$ that its minimum in $\left[\frac{1}{2n}, \frac{1}{2}\right]$ occurs at a rational number (or possibly more than one).  Therefore, it suffices to show that there is a constant $c$, independent of $n$, such that $f_n(\alpha) \geq cn$ for all rational $\alpha \in \left[\frac{1}{2n}, \frac{1}{2}\right]$.  In what follows, we will be rather wasteful with our estimates, but for clarity we will produce explicit constants at each step.

Naturally,
$$
f_n\left(\frac{1}{2}\right)\geq \left\lfloor \frac{n}{2} \right\rfloor \frac{1}{4} \geq \frac{n}{16}.
$$
Now let $\alpha= \frac{a}{b}$ with $\gcd(a,b)=1$ and $3 \leq b \leq n$.  For each $j \in [1,b-1]$, we have
$$
\#\{k \leq n: \ ka \equiv j \pmod{b}\} \geq \left\lfloor\frac{n}{b} \right\rfloor.
$$
Thus,
$$
f_n(\alpha) = \sum_{k \leq n} \left\|k \frac{a}{b} \right\|^2 \geq  2 \cdot \sum_{j < \frac{b}{2}} \left\lfloor\frac{n}{b} \right\rfloor \frac{j^2}{b^2}      \geq 2 \cdot \left\lfloor \frac{n}{b} \right\rfloor \frac{1}{b^2} \frac{b^3}{2 \cdot 6 \cdot 2^3} \geq \frac{1}{96} n.
$$

Now assume $b > n$, $\gcd(a,b)=1$, and $\frac{1}{2n} \leq \frac{a}{b} \leq \frac{1}{2}.$  If $\frac{b}{2} \leq na < b$, then clearly
$$
f_n(\alpha) = \sum_{k \leq n} \left\|k\frac{a}{b}\right\|^2 \geq \frac{a^2}{b^2}\sum_{k \leq \frac{n}{2}} k^2 \geq \frac{1}{2^2n^2}\frac{n^3}{2\cdot 6 \cdot 2^3} =\frac{n}{384}.
$$

Now assume $na \geq b $ and note that $a$ generates the additive group $\pmod{b}$.  Partition the set $\{ka\}_{k=1}^n$ into subsets between multiples of $b$ as
$$
\left\{a, 2a, \dots, \left\lfloor \frac{b}{a} \right\rfloor a\right\} \cup \left\{\left(\left\lfloor \frac{b}{a} \right\rfloor +1\right)a, \dots, \left(2\left\lfloor \frac{b}{a} \right\rfloor + \eta_2\right)a\right\} \cup \dots,
$$
where, in the $j$-th set, $\eta_j \in \{0,1\}$.  There are at least $\floor[\big]{n / \floor[\big]{\frac{b}{a}}} \geq 1$ such sets, and each contains a sequence of $\floor[\big]{\frac{b}{a}} \big/ 2 \geq 1$ elements that are at least $a, 2a, \dots, \floor[\big]{\frac{b}{a}}a \big/ 2,$ respectively.  Hence, we have
$$
f_n(\alpha) \geq \left\lfloor \frac{n}{\left\lfloor \frac{b}{a} \right\rfloor} \right\rfloor \sum_{j \leq \left\lfloor \frac{b}{a} \right\rfloor \big/ 2} \frac{j^2a^2}{b^2} \geq \frac{n}{2\left\lfloor \frac{b}{a} \right\rfloor}\frac{a^2}{b^2} \frac{\left\lfloor \frac{b}{a} \right\rfloor^3}{2\cdot 6 \cdot 2^3} \geq \frac{n}{768},
$$
since $\frac{b}{a} \geq 2$ implies $\frac{a}{b}\left\lfloor \frac{b}{a} \right\rfloor \geq \frac{1}{2}$.  Thus, \eqref{E:inffn} is proved and with it, Lemma \ref{L:rothszekeres}.
\end{proof}


\begin{thebibliography}{99}

\bibitem{AS} Abramowitz, M., Stegun, I.: Handbook of Mathematical Functions with Formulas, Graphs, and Mathematical Tables. Dover (1972)

\bibitem{A}  Andrews, G.: The Theory of Partitions. Cambridge Univerity Press (1984)

\bibitem{Bi}  Billingsley, P.: Probability and Measure.  John Wiley \& Sons Inc. (1995)

\bibitem{B} Bridges, W.: Limit shapes for unimodal sequences, preprint (2020); available at \href{https://arxiv.org/abs/2001.06878}{arXiv:2001.06878}

\bibitem{DFLS} Duchon, P., Flajolet, P., Louchard, G., Schaeffer, G.: Boltzmann samplers for the random generation of combinatorial structures,  Combin. Probab. Comput.  13, 577-625 (2004)

\bibitem{F}  Fristedt, B.: The structure of random large partitions of integers, Trans. Amer. Math. Soc. 337, 703-735 (1993)

\bibitem{FS}  Flajolet, P., Sedgewick, R.:  Analytic Combinatorics. Cambridge Univeristy Press, (2009)

\bibitem{HR} Hardy, G.H., Ramanujan, S.: Asymptotic formul$\ae$ in combinatory analysis,  Proc. London Math. Soc. (2)  17, 75-115, (1918)

\bibitem{MV}  Montgomery, H.L., Vaughn, R.C.: Multiplicative Number Theory: I. Classical Theory. Cambridge University Press (2006)

\bibitem{R}  Romik, D.: Partitions of $n$ into $t\sqrt{n}$ parts,  European J. Combin.  26, no. 1, 1-17 (2005)

\bibitem{RS} Roth, K., Szekeres, G.: Some asymptotic formul$\ae$ in the theory of partitions, Quart. J. Math. Oxford Ser. (2) 5, 241-259, (1954)

\bibitem{S1} Szekeres, G.: An asymptotic formula in the theory of partitions, Quart. J. Math. Oxford Ser. (2) 2, 85-108 (1951)

\bibitem{S2} Szekeres, G.: Some asymptotic fomul$\ae$ in the theory of partitions II,  Quart. J. Math. Oxford Ser. (2) 4, 96-111 (1953)


\end{thebibliography}
\end{document}